\documentclass[12pt]{extarticle}
\usepackage{standalone}
\usepackage{amsmath}
\usepackage{amssymb}
\usepackage{amsbsy}
\usepackage{amsfonts}
\usepackage{eucal}
\usepackage{titlesec}
\usepackage{caption}
\usepackage{tikz-cd}
\usepackage[all,cmtip]{xy}
\usepackage{concrete}
\usepackage{amsthm}
\usepackage{mathrsfs}
\usepackage{geometry,graphicx,color}
\usepackage{mathtext}
\usepackage{hyperref}
\usepackage{fancybox}
\theoremstyle{plain}
\newtheorem{theorem}{Theorem}
\newtheorem{lemma}[theorem]{Lemma}
\newtheorem{prop}[theorem]{Proposition}

\newtheorem{cor}[theorem]{Corollary}

\newcommand{\xar}[1]{\ensuremath{\xrightarrow{#1}}}
\newcommand{\xlar}[1]{\ensuremath{\xleftarrow{#1}}}

\newcommand{\mc}[1]{\mathcal{#1}}
\newcommand{\mr}[1]{\mathrm{#1}}

\newcommand{\ms}[1]{\mathscr{#1}}

\newcommand{\Q}{\mathbb{Q}}
\newcommand{\Z}{\mathbb{Z}}

\newcommand{\on}[1]{\operatorname{#1}}
\newcommand{\la}{\langle}
\newcommand{\ra}{\rangle}

\newcommand{\op}{\mathrm{op}}

\usepackage{chngcntr}
\counterwithout{subparagraph}{paragraph}
\newcommand{\p}[1]{\subparagraph{#1}}
\titleformat{\subparagraph}[runin]{\normalsize\bfseries}{\textbf{(\arabic{subparagraph}})}{0.5em}{}[]

\titleformat{\section}[block]{\normalsize\bfseries}{\textbf{\arabic{section}.}}{0.5em}{}[]

\title{A note on a local combinatorial formula for the Euler class of a PL spherical fiber bundle}
\author{Nikolai Mn\"ev\thanks{PDMI RAS;  Chebyshev Laboratory, SPbSU. Research is supported by the Russian Science Foundation grant 19-71-30002.} \\  \href{mailto:mnev@pdmi.ras.ru}{mnev@pdmi.ras.ru} }

\begin{document}
\maketitle
\begin{abstract} We present a local combinatorial formula for the Euler class of a $n$-dimen\-si\-onal PL spherical fiber bundle as a rational number  $e_{\it CH}$ associated to a chain of $n+1$ abstract subdivisions of  abstract  $n$-spherical PL cell complexes. The number $e_{\it CH}$ is a combinatorial (or matrix) Hodge theory  twisting cochain in Guy Hirsch's  homology model of the bundle associated with PL combinatorics of the bundle. \end{abstract}

\tableofcontents

\section{Introduction.}
\p{The problem on finding a local combinatorial formula for the Euler class of a spherical fiber bundle  in PL category.} \label{slf}
PL category is defined using triangulations. A spherical PL fiber bundle $S^n\xar{}E \xar{p} B$  has a triangulation given by a map of simplicial complexes $\pmb E \xar{\pmb p} \pmb B$. We call the stalk of the triangulation over a base simplex $\sigma^k \in \pmb B(k)$  an elementary simplicial spherical bundle. It triangulates the trivial bundle $S^n \times \Delta^k \xar{\pi} \Delta^k $ in such a way that any simplex in the total space is mapped onto a face of the base simplex (see Figure \ref{eb}).
\begin{figure}[h!]
	\begin{center}
		\includegraphics[width=3.0in]{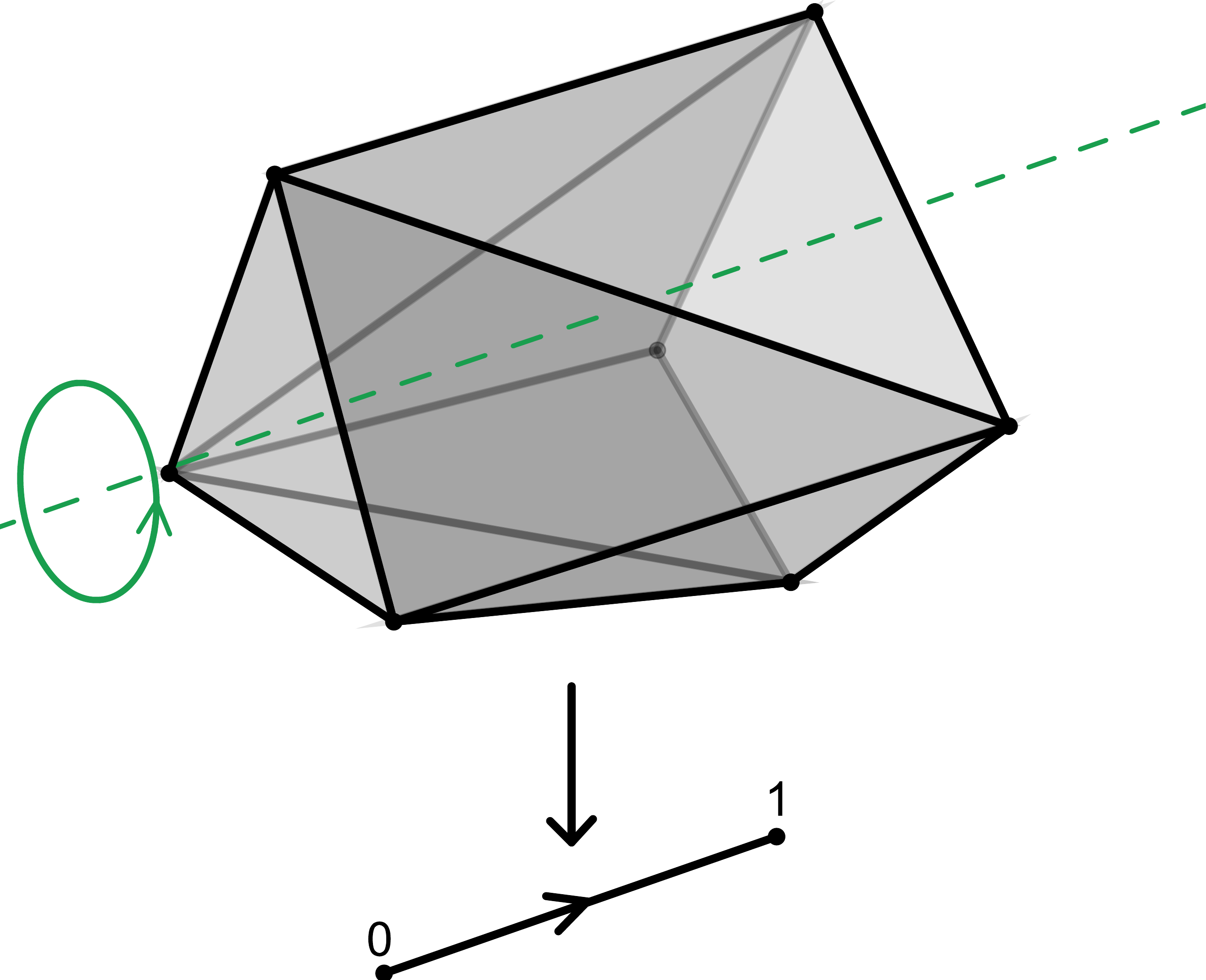} 
		\caption{ Elementary simplicial circle bundle \label{eb}}
	\end{center}
\end{figure} The  triangulation  $\pmb E \xar{\pmb p} \pmb B$  is assembled from elementary simplicial spherical bundles using boundary maps - combinatorial  automorphisms of the elementary bundles.
Suppose that the base simplicial complex is locally ordered. Therefore it has a complex of  ordered cochains, which  compute singular the  cohomology of the base.     We wish to find a universal  rational function of the combinatorial isomorphism class of elementary triangulated oriented $S^n$-bundles over an $(n+1)$-simplex  such that this value being assigned to the base $(n+1)$-simplex is a rational cocycle representing the Euler class of the bundle. Thus it is supposed to be independent of boundary combinatorial automorphisms composing a bundle from elementary bundles. Such a formula we call \textit{a simplicial local combinatorial formula for the Euler class}.

Since the Euler class in an integer characteristic class, the rational formula should  have integer periods, now in the combinatorial setting. I.e. its evaluation on integer $(n+1)$-simplicial cycles in the base are integer numbers depending on the isomorphism class of the bundle and homology class of the cycle, and independent of triangulation. Particularly, if we a triangulate differentiable  $S^n$ bundle on a differentiable closed oriented $(n+1)$-base, we should obtain the same Euler number of the bundle out of combinatorial and out of differential considerations.    Therefore the arithmetic  of the  formula is highly non-trivial.

\p{} There is an alternative to triangulations and somewhat dual combinatorics of PL fiber bundles investigated in \cite{Mnev:2007}.  

By a \textit{local system} $\mathcal G$ on $\pmb B$ with values in some category $\pmb G$ we mean a map associating to any vertex $v$ of $\pmb B$ an object $G(v)$ of $\pmb G$ and  to any oriented  edge $(v_0,v_1)$ a $\pmb G$-morphism  $G(v_0) \xar{G(v_0,v_1)} G(v_1)$, in a way that  any 2-simplex $(v_0,v_1,v_2)$  goes to composition $G(v_0,v_2)=G(v_1,v_2)G(v_0,v_1)$. Alternatively we may say that we have a simplicial map $\pmb B \xar{\mathcal G} \mathscr N \pmb G$, where $\mathscr N$ denotes nerve of category. 

One may encode a spherical PL fiber bundle $p$  by a local system on the base simplicial complex  $\pmb B$ with values in the category $\pmb S^n$ of abstract regular spherical PL cell complexes and corresponding  aggregations (this explained in Sec. \ref{pl}). Any triangulation has such a local system canonically associated. For this combinatorics a local formula for the Euler class of a PL $S^n$-bundle is a universal Euler  $(n+1)$-cocycle which is a function a of chain of $n+1$ subdivisions (or aggregations) of $n$-spherical cell complexes, e.g. a chain of subdivisions of convex polytopes. This function measures certain combinatorial asymmetry in the chain. We call it \textit{the aggregation local combinatorial formula for Euler class} and it is defined in Sec. \ref{pl} \S \ref{aggreu}.          
        
\p{}This note is devoted to a simple observation: the combinatorial
model of spherical  PL fiber bundle as a local system of abstract aggregations (or subdivisions) of PL spherical cell complexes on a triangulated base (\cite{Mnev:2007}) smoothly and naturally fits into the very classical theory of Hirsch homology models of fibrations, since it produces exactly the ``bigraded model of a fibration" and its Serre spectral sequence is the bicomplex spectral sequence. Thus, we are immediately producing a local combinatorial formula for Euler class as a ``twisting cochain". In fact, we substitute the combinatorics of cellular local systems from \cite{Mnev:2007} into the deformation theory of local systems of spherical chain complexes as expressed in \cite[Corollary 2.5]{Igusa2011} with explicit  formula \cite[(3)]{Igusa2011} for the Euler cocycle. The simplicial local combinatorial formula in this setting is derived from the formula for local systems. The resulting aggregation  local formula (Theorem \ref{finformula}) is composed from combinatorial  Hodge-theoretic retractions of chain cellular spheres in homology. This is expected  to be an interesting subject of combinatorics, statistics, thermodynamics etc of ``Higher Kirchhoff theorems" \cite{Lyons2009,CCK2012,DKM2015,CCK2015,CCK17}.
 
\bigskip
\noindent
Unless stated otherwise we assume that our coefficients are in some characteristic 0 field $A$, since our goal is a rational formula. 
\section{Euler class of an oriented spherical fiber bundle.}
\p{Gysin homomorphism and trangression differential.}
An oriented spherical fiber bundle
\begin{equation}\label{bund}
	S^n \xar{} E\xar{p} B
\end{equation} has integer Euler characteristic class
$\mathscr E(p) \in H^{n+1}(B;\Z)$.
In    \cite{Steenrod1949, CS1950,Thom1952} the Euler class of fiber bundle (\ref{bund}) was identified via Gysin homomorphism $G= \ms \frown \ms E(p)$ of homological Gysin exact sequence of the bundle
\begin{multline} \label{gysin}
	\cdots \xar{p} H_{n+k+1} (B;\Z) \xar{\frown \ms E(p)} H_k(B;\Z) \xar{j} H_{n+k}(E;\Z)\xar{p}H_{n+k}(B;\Z)\xar{\frown \ms E(p)} \cdots \\
	\cdots \xar{\frown \ms E(p)} H_1(B;\Z)\xar{j}H_{n+1}(E;\Z)\xar{p} H_{n+1}(B;\Z) \xar{\frown \ms E(p)} H_0(B;\Z) \xar{}0	
\end{multline}	
Gysin homomorphism is the differential on $(n+1)$-st page of Leray-Serre spectral sequence of the bundle
\begin{equation*}\label{}
	H_{n+k+1}(B; \underset{\approx \Z}{H_0(S^n;\Z)})\approx E^{n+1}_{n+k+1,0}(p) \xar{d_{n+1} = \frown\ms E(p) }  E^{n+1}_{k,n}(p)\approx H_k(B, \underset{\approx \Z}{H_n (S^n;\Z)})
\end{equation*}
And thus the first one - the transgression differential
\begin{equation}\label{trans}
	H_{n+1}(B; \underset{\approx \Z}{H_0(S^n;\Z)})\approx E^{n+1}_{n+1,0}(p) \xar{\frown\ms E(p) }  E^{n+1}_{0,n}(p)\approx H_0(B, \underset{\approx \Z}{H_n (S^n;\Z)})
\end{equation}
-- can be considered as  Euler class itself if the base is connected.
\section{Guy Hirsch model of a fibration.}  Guy Hirsch in \cite{Hirsch1954} introduced ``Hirsch homology models of fibrations."  For certain  fibrations $F\xar{} E \xar{g} B$ he detected a subcomplex of the chain complex $C_\bullet (E) $ of the form $C_\bullet (B)\otimes H_\bullet (F)$ which has the same homology as  $E$. In \cite{Brown1959} E. Brown  defined the ``twisted tensor product" and recognized Hirsch model in the form of twisted Eilenberg-Zilber theorem, $C_\bullet (E) \approx C_\bullet (B)\otimes_{\tau(g)} H_\bullet(F)$, where $\tau(g)$ is ``twisting cochain."

Twisting cochains for Hirsch models allow chain-level understanding of the differentials in the Serre spectral sequence of a Serre fibration. This was investigated in detail by Georgian school \cite{Berika1976, Berikashvili2006,Kadei1986}.  Modern
setup for Brown's  twisting cochain has the form of $A_\infty$ local systems \cite[Sec. 1,2]{Igusa2011}.

 There are two steps in the construction of Hirsch model  of a fibration $F \xar{} E \xar{g} B$: an algebraic step  and a topological step.

 In the case of a base polyhedron $B$ (\cite{Berika1976,Kadei1986,Igusa2011} ), the  algebraic step
 investigates a local system $\mathcal L$ of chain complexes with fixed homology $H_\bullet (F) = \sum_{k=0}^{n} H_k (F)$ on a locally ordered triangulation of the base $\pmb B, |\pmb B| = B$.  Let    $\on{Tot} \mathcal L$ be the naturally filtered total complex of the bicomplex $C_\bullet(\pmb B; \mc L_\bullet)$ of  $\mc L$.  The Hirsch-Brown model  of $\mc L$ is the deformation of the differential in the complex   $C_\bullet (\pmb B) \otimes H_\bullet(F) \rightsquigarrow C_\bullet (\pmb B) \otimes_{\tau(\mathcal L)} H_\bullet(F) $ using a  \textit{twisting cochain} $\tau(\mathcal L)$ so that  $\on{T ot} \mathcal L$ is filtred homotopy equivalent to $C_\bullet (\pmb B) \otimes_{\tau(\mathcal L)}  H_\bullet(F)$ producing an equivalent spectral sequence with differentials readable on the chain level.

 In the next topological step we try to replace the fibration $g$ over the polyhedron $|\pmb B|$ by a local system of\textit{ spaces} 
 $\mc W$ over $\pmb B$ in such a way that certain chain complexes $C_\bullet (\mc W_v)$ of the local system will give rise to the ``bigraded model of the fibration" - a bicomplex with total complex equivalent to the total singular complex of the fibration  (\cite{Dress1967,Fadell1958,Berikashvili2006}).  For this brigaded model the algebraic step became applicable and we obtain $ C_\bullet(\pmb B) \otimes_{\tau(\mathcal W(g))}  H_\bullet (F)$ as the Hirsch model of $g$ and thus we can see chain-level formulas for  differentials in the Serre spectral sequence of $g$ up to irrationalities in the construction of the bigraded model.
\section{Local systems of spherical chain complexes on a simplicial base, twisting cochain and formal Euler cocycle.}  We present in a suitable for us form the classic construction (\cite{Berika1976,Kadei1986,Igusa2011}) of the twisting cochain for  a local system of spherical chain complexes on a simplicial base.  Our coefficiens are in a characteristic $0$ field $A$.
 \p{Local systems of spherical chain complexes.} \label{locsys} Let $\on{Ch}(S^n)$ be the category of oriented spherical chain complexes.
Objects are length $n+1$  chain complexes of finite-dimensional vector spaces over $A$, where we denote the differentials by $\gamma$
$$K_\bullet =(0\xar{}K_n \xar{\gamma_n}K_{n-1} \xar{\gamma_{n-1}}\cdots \xar{\gamma_1} K_0 \xar{} 0) $$ We suppose that that $H_i(K) = A$ if $i=0,n$ and $H_i(K)=0$ otherwise. Also we suppose that $K$ has fixed augmentation denoted by $K_0 \xar{p_0} A$ and fixed ``orientation"  $i_n$ making the sequence exact:
$$0\xar{}A \xar{i_n} K_n\xar{\gamma_n} K_{n-1}$$ Orientation fixes ``fundamental class" $i(1)\in K_n$ considered as generator  in $H_n(K)$.  Morphisms in $\on{Ch} (S^n)$ are degree $0$ chain maps commuting with augmentation and orientation. We have a special homology spherical complex with zero differentials
\begin{equation}\label{key}
	 H_\bullet(S^n) = (0 \xar{} \underset{0}{A} \xar{} 0 \xar{} \cdots \xar{}0 \xar{} \underset{n}{A} \xar{}0)
\end{equation}

Let $\pmb B$ be a locally ordered simplicial complex equipped with  a  $\on{Ch}(S^n)^\op$-valued local system $\mc L$ on $\pmb B$  defied by simplicial map $ \pmb B \xar{\mc L} \ms N \on{Ch(S^n;A)^\op}$.
Associating  to a simplex $\sigma_p = (v_0,...,v_p) \in \pmb B(p)$
the complex $L(v_p)$ sitting over its  { \em last vertex} and to i-th face inclusion $d_i \sigma_p \xar{} \sigma_p $ the identity map of complexes if $i=0,..., p-1$ and the map $L(v_{p-1},v_p)$ if $i=p$, we obtain from the simplicial local system a simplicial constructible sheaf on $\pmb B$.
\p{Leray-Gysin spectral sequence of a spherical local system.}
We can consider the complex $C_\bullet (\pmb B ; \mc L_\bullet)$ of simplicial chains on $\pmb B$ with coefficients in $\mc L$. It is a bigraded module with two anticommuting differentials. The  module $C_p(\pmb B, \mc L_q)$ is formed by $(p,q)$ chains $\sigma_p c_q$ assigning  to  a simplex $\sigma_p =(v_0,...,v_p)\in \pmb B(p)$ an element $c_{q}\in L_q(v_p)$, i.e a $q$-element of the complex over  the last  vertex $v_p$. We have two anticommuting differentials: the simplicial horizontal differential
\begin{equation}\label{key}
	\begin{array}{c} C_p(\pmb B, \mc L_q)\xar\partial C_{p-1}(\pmb B, \mc L_q) \\
		\partial (\sigma_p c_{q}) = \sum_{i=0}^{p-1} (-1)^i (d_i \sigma_p) c_q + (-1)^p (d_p \sigma_p) L(v_{p-1},v_{p})(c_q)
	\end{array}
\end{equation}
where the non-trivial transition map appears only in the last summand
and the vertical differential
\begin{equation}\label{key}
	\begin{array}{c}
		C_p(\pmb B, \mc L_q)\xar{\tilde \gamma} C_{p}(\pmb B, \mc L_{q-1}) \\
		\tilde \gamma (\sigma_p c_q) = (-1)^p\sigma_p \gamma_{\sigma_p} (c_q)
\end{array}  \end{equation}	
induced by the differential in $Ch(S^n)$. Thus we obtained the total complex $\on{Tot}(\pmb B, \mc L )$ with
total differential  \begin{equation}\label{key}
	\on{Tot}= \partial + \tilde \gamma
\end{equation}
and horizontal filtration $$F_p (\on{Tot} C_\bullet  (\pmb B; \mc L_\bullet))_d=\oplus_{{\underset{k_1 \leq p}{k_1 + k_2}}  = d} C_{k_1}(\pmb B; \mc L_{k_2}) $$
Its first quadrant spectral sequence $E^\bullet_{\bullet, \bullet}$
staring from  page zero converges to the chain homology of the $\on{Tot}$-complex.  On  page zero we have: $$\begin{tikzcd}
	{E^0_{p,q}=C_p(\pmb B; \mathcal L_q)} \arrow[d, "d_0 = \tilde \gamma", shift right=10] \\
	{E^0_{p,q}=C_p(\pmb B; \mathcal L_{q-1})}
\end{tikzcd}$$
On $E^1_{\bullet \bullet}$ the 1-differential is horizontal differential $\partial$
$$C_{p-1} (\pmb B;  H_q(S^n)) = E^1_{p-1,q} \xlar{d_1 = \partial} E^1_{p,q}= C_p (\pmb B;  H_q (S^n))$$
I.e. $E_{p, q}^1 \approx C_p (\pmb B;  H_q (S^n))$. Thus $E^2_{p,q} = H_p(\pmb B;  H_q (S^n))$. On page $n+1$ we get Gysin-Leray transgression differential (\ref{trans}) as the \textit{formal} Gysin homomorphism. ``Formal," because  a priori the local system $\mc L$ it is not related to any spherical Serre fibration on $|\pmb B|$.

\p{Brown's twisting cochain as a formal Euler cocycle.} We need some objects and notions. $C_\bullet(\pmb B)$ has Alexander-Whitney  coalgebra structure. Let $H=H_\bullet= \sum_i H_i$ some graded module. Then $\on{Hom} (H,H)$ has an algebra structure by composition and we can consider the DGA  $C^\bullet (\pmb B; \on{Hom}(H,H))$ with Alexander-Whitney product.
A \textit{twisting cochain} $\tau$ is a cochain \begin{equation}\label{tc}
	\tau = \tau^1 + \tau^2 + \dots, \tau^i \in C^i(\pmb B; \on {Hom}^{i-1} (H,H) )
\end{equation}
satisfying the condition
\begin{equation}\label{tc1}
	\delta \tau = -\tau \smile \tau
\end{equation} There is a pairing $ H \otimes \on{Hom}(H,H) \xar{}  H$ sending $e_q \in  H_q$ and $f_i \in \on{Hom}_i(H,H)$ to element $e_q  f_i = f_i (e_q) \in  H_{q+i}$. We have the cap product $C_k (\pmb B;  H) \otimes C^p (\pmb B;  \on{Hom}( H,  H) \xar{\frown} C_{p-k}(\pmb B,  H)  $.
Having a twisting cochain $\tau$ we can deform the differential as \begin{equation}\label{tcdiff}
	\partial \rightsquigarrow \partial \; + \; \frown \tau
\end{equation} on $C_\bullet (\pmb B)\otimes H$ thus obtaining a new filtered DG module $C_\bullet (\pmb B)\otimes_\tau H$. The new differential respects the horizontal filtration on $C_\bullet (\pmb B) \otimes  H$ therefore it just adds some nontrivial differentials to the trivial bicomplex   spectral sequence in the simplest possible way.

Now let $H=H(S^n)$. Let the twisting cochain has $\tau^1 =0$. In this situation  a twisting cochain $\tau $ on $C_\bullet(\pmb B)\otimes H(S^n)$ is just \textit{formal} \textit{Euler cocycle}. Indeed, by dimension reason equation (\ref{tc1}) became the cocycle condition $\delta \tau =0$, thus $\tau$ has single  non-zero element $\tau^{n+1} \in C^{n+1}(\pmb B; \on{Hom}^n(H(S^n),H(S^n))$.   The spectral sequence of horizontal filtration becomes the Gysin-Leray sequence with transgression differential $\; \frown \tau$ on page $n+1$. So $\tau(\sigma_{n+1})(1)$ is nothing but an $A$ - valued ``Euler" simplicial  $n+1$-cocycle on the base.
\p{Strong deformational retractions, Basic Homology Perturbation Lemma.}
 Recall the notions of strong deformational retraction and homology perturbation.
Let $C, K$ be chain complexes of modules over a commutative ring with unit. Strong deformation  retraction (``SDR") (see \cite{Gugenheim1989}) of $ C$ on $K$ is the data $\la {F,i,p} \ra$ of the diagram of chain
maps:
\begin{equation}\label{sdr}
	\begin{tikzcd}
		C \arrow[r, "p", shift left] \arrow["F"', loop, distance=2em, in=215, out=145] & K \arrow[l, "i", shift left]
	\end{tikzcd}
\end{equation}
Here the retraction operator $C \xar{\mr F} C[1]$ shifts the dimension by one and the following conditions holds:
\begin{equation}\label{retr}
	pi=Id; dF+Fd=Id - ip
\end{equation}
\begin{equation}\label{nil}
	Fi=0;pF=0;F^2=0
\end{equation}
Here the annihilation  conditions (\ref{nil}) can be satisfied if (\ref{retr}) holds by perturbing $\mr F$. In particular $i$ splits the exact sequence $$ 0\xar{} \on{Ker} p \xar{} C \xar{p} K \xar{} 0$$ with  projection on the kernel given by $dF+Fd$ and thus represents $C$ as a direct sum of $K$ and a contractible $\on{Ker} p$.

Let  $H(C)$ be the homology complex of $C$ considered as a complex  with zero differentials. The complex $C$ is called {\em homology split} if there exist a {\em homology splitting} (see for ex. \cite[\S 1]{Gugenheim1982}), which is an SDR of $C$ on $H(C)$:
$$
\begin{tikzcd}
	C \arrow[r, "p", shift left] \arrow["F"', loop, distance=2em, in=215, out=145] & H(C) \arrow[l, "i", shift left]
\end{tikzcd} $$
representing $C$ as a direct sum of its homology module and a trivial chain complex. If the complex and its homology are free then a homology splitting exist. In particular if $C$ has trivial
homology then its homology splitting is  a degree 1 contraction $C\xar{F} C[1]$ such that $F^2 =0$ and $dF+Fd = \on {Id}$.

If complexes $C,K$ are filtered and SDR data preserves filtrations, then SDR is {\em filtered}.

A perturbation of SDR data is a degree $-1$
homomorphism $C \xar{\psi} C[-1] $ such that $(d_C + \psi)^2 =0$, i.e. $d_C + \psi$ also a differential. 
The fundamental tool to obtain otherwise unavailable formulas is following
\begin{lemma}[Basic Perturbation Lemma \cite{Gugenheim1989} ] \label{bpl}
	Let $\la {F,i,p} \ra$ be a filtered SDR (\ref{sdr}) and $\psi$ its perturbation.
	Then
	$$
	\begin{tikzcd}
		C_{d_C +\psi} \arrow[r, "p_\psi", shift left] \arrow["F_\psi"', loop, distance=2em, in=215, out=145] & K_{d_K + d_\psi} \arrow[l, "i_\psi", shift left]
	\end{tikzcd} $$
	is a filtered SDR
	where
	\begin{align}
		d_\psi & = p\psi\Sigma^\psi i \label{diff} \\
		p_\psi & = p(1- \psi \Sigma^\psi_i F )\\
		i_\psi & =  \Sigma^\psi i \\
		F_\psi & =  \Sigma^\psi F 	
	\end{align}	
	Here
	$$\Sigma^\psi = \sum_{j \geq 0} (-1)^i (F\psi)^j = 1 - F\psi + F\psi F\psi - F\psi F\psi F\psi \cdots  $$
	
\end{lemma}

\p{Hirsch model of a $\on{Ch}(S^n)$-local system with retractions on homology, its twisting cochain and a formal Euler cocycle.} Here we present a deformation of $\on{Ch}(S^n)^\op$ local system $\mc L$ onto the trivial local system $C_\bullet(\pmb B; H(S^n)) $ by a locally defined formal Euler cocycle. The deformation and the cocycle determined by a certain free extra structure - retractions onto homology of complexes  $L_v, v \in \pmb B(0)$.  Suppose that for
any spherical complex $L_v$ in the local system  we have fixed fixed
strong deformational retraction on homology $H(S^n)$. This will be the extra data $\mc F$:
\begin{equation}\label{ret}
	\begin{tikzcd}
	L_v \arrow[r, "p_v", shift left] \arrow["F_v"', loop, distance=2em, in=215, out=145] & H(S^n) \arrow[l, "i_v", shift left], v \in \pmb B (0)
\end{tikzcd} 
\end{equation}

\begin{prop}[Corollary 2.5 and formula (3) in \cite{Igusa2011}] \label{Igu}
There is a filtering preserving  strong deformational retraction
of $\on{Tot}(\pmb B; \mathcal L) $ onto $C_\bullet (\pmb B) \otimes_{\mathscr E (\mc F)} H(S^n)$ where \begin{multline} \label{formal} \mathscr E (\mc F)(v_0,...,v_{n+1}) = \\ p_n L_{0,1}F_1L_{1,2}....F_nL_{n,n+1}i_0 \\ \in Z^{n+1}(\pmb B; \on{Hom}^n (H(S^n),H(S^n))) \end{multline}
is the Euler cocycle for the local system of spherical chain complexes  $\mc L$ endowed with retractions on homology $\mc F$.
\end{prop}
\begin{proof} \mbox{} \\
\indent
	(1)	 This is a typical application of the Basic Homology Perturbation Lemma \ref{bpl}. Let is forget about the horizontal differential $\partial$ in $\on{Tot}(\pmb B;\mc L)$ and obtain the complex $\on{Tot}^{\tilde{\gamma}} (\pmb B; \mc L)$. The family of retractions $\mc F$ (\ref{ret}) provide filtred SDR of  $\on{Tot}^{\tilde{\gamma}} (\pmb B, \mc L)$ on filtred module $ C^0_\bullet (\pmb B)\otimes H(S^n)$ with zero differentials. Now we perturb the differential $\tilde{\gamma} \rightsquigarrow \tilde{\gamma}+\partial $ restoring initial differential in $\on{Tot} (\pmb B; \mc L)$ and see what happens by the lemma identities.
Compute the new differential $d_\partial = p\partial \Sigma^\partial i $ on the element $\sigma^k h \in C^0_k(\pmb B)\otimes H(S^n)$. Directly applying annihilation conditions (\ref{nil}) we get an expression:
$$d_\partial (\sigma^kh) = \partial \sigma^kh + \sigma^k \frown \mathscr E (\mc F)(v_0,...,v_{n+1})(h)$$     	
	
(2)  Check that $\mathscr E (\mc F)(v_0,...,v_{n+1})$ is a cocycle.
Consider the bigraded module of simplicial cochains $C^\bullet (\pmb B; \mc L_\bullet)$ with values in the local system.
Put $c^p_q(\sigma_p) \in \mc L_q (v_0 \sigma_p)$ -- we assign an element over the { \em first} vertex. It has anticommuting  horizontal codifferential
\begin{equation}\label{codiff}
	\delta (x (\sigma_{p})) = L(v_0,v_1)(x(d_0 \sigma_p)) + \sum_i^p (-1)^ix (d_i \sigma_p)  \in L (v_0 (\sigma_p) )
\end{equation}
 and vertical differential $\tilde \gamma$.
Introduce a new  local system $\on {Hom}_\bullet (H; L_v)$ on $\pmb B$. Then we have a system of cochains
\begin{equation}\label{tt}
	\begin{array}{rl}
		i \in & C^0 (\pmb B; \on{Hom}_0(H; L_v) )\\
		\delta i \in & C^1 (\pmb B; \on{Hom_0}(H; L_v))\\
		F\delta i \in & C^1 (\pmb B; \on{Hom_1}(H; L_v))  \\
		\delta F \delta i \in &  C^2 (\pmb B; \on{Hom}_2(H;L_v)) \\
		\cdots & \cdots \\
		(F\delta)^n i \in & C^n (\pmb B; \on{Hom}_n(H ; L_v)) \\
		\delta (F\delta)^n i \in  & C^{n+1} (\pmb B; \on{Hom}_n(H ; L_v))
	\end{array}
\end{equation}
Differentials  $\tilde \gamma $ and $\delta$ anticommute.
We can prove
inductively  by $k$ that the cochain $\delta (F\delta)^k i$ is a cycle for $\tilde \gamma$.
The inductive step is
$$\tilde \gamma \delta (F\delta)^k i = - \delta \tilde \gamma F \underbrace{\delta (F\delta)^{k-1}i}_{\tilde \gamma - \text{ cycle }}= -\delta \delta (F\delta)^{k-1} =0$$
since by (\ref{retr}) if $\tilde \gamma x =0$ then $\tilde \gamma F x = x$.
Therefore, $\delta (F\delta)^n i$ is a $\tilde \gamma$-cycle and therefore it is  proportional to the fundamental class $i_n (1)$. Simultaneously it is a $\delta$-cocycle being a $\delta$-coboundary. Therefore $p_n \delta (F\delta)^n i_0 \in Z^{n+1}(\pmb B; \on{Hom}(H;H))$ is a $\delta$-cocycle.

Now we compute (\ref{tt}) using (\ref{codiff}).
$$\delta i (v_0,v_1) = L(v_0,v_1)i(v_1) - i(v_0) \in L_{v_0}(0) $$
$$ F \delta i = F_{v_0} L(v_0, v_1) i (v_1)  - \underbrace{F_{v_0} i (v_0)}_{=0 \text{ since } Fi=0}$$
\begin{multline*}
\delta F \delta i (v_0,v_1,v_2) = \\ L(v_0,v_1) F(v_1) L(v_1,v_2)i(v_2) - F(0) (L(0,1)i(1) +  L(0,2) i(2)) \end{multline*}
$$(F\delta)^k i (0,1,...,k) = F(0)L(0,1)F(1)... L(k,k-1) i(k) +\underbrace{F^2(0)(...)}_{=0} $$
\begin{multline*}
 \delta (F\delta)^k i (0,...,k+1) = \\ L(0,1) F(1)L(1,2)F(2)... L(k+1,k) i(k+1) +  F(0)\sum_{i=1}^{k+1} (-1)^i (...)
\end{multline*}
Finnaly since $p(0)F(0)=0$ we get that
\begin{multline*}
p(F\delta)^n i (0,..., n+1) = \\ p(0)L(0,1)F(1)L(1,2)... F(n)L(n,n+1) i(n+1)	= \\ \mathscr E (\mc F)(v_0,...,v_{n+1})
 \end{multline*} and therefore $ \mathscr E (\mc F)(v_0,...,v_{n+1})$ is a cocycle and an Euler twisting cochain for the  Hirsch model of the spherical local system $\mc L$.
\end{proof}

\section{PL combinatorics of a spherical fiber bundle} \label{pl}
\p{Geometric and abstract regular spherical PL cell complexes.} 
Let $S^n$ be the $n$- dimensional sphere in PL category. A \textit{regular geometric cell complex} structure $B$ on $S^n$ is a covering of $S^n$ by collection of  closed embedded PL balls $B$ such that interiors of the balls form a partition of $S^n$ and the boundary of a ball is a union of  balls. The face complex of a convex polytope  is the most obvious example.  The other names for these objects in the literature are ``ball complexes" (\cite[Appendix to Ch 2 (5) ]{RS1972}) or ``regular CW complexes" (\cite[Ch III, sect. 1,2]{LW}).
Partially ordered by inclusion, the set $P(B)$ of balls in $B$ defines $B$ up to PL homeomorphism (see \cite{Bj84}). Thus we can define an \textit{abstract regular spherical  PL cell complex} as a finite poset $P$ for which the simplicial order  complex $\Delta P$ is PL-homeomorphic to $S^n$ and all the principal lower ideals are PL homeomorphic to balls.

\bigskip
\noindent
\textit{Unless specified otherwise, all our cell complexes are regular and PL.}

\bigskip
\noindent
If we have a poset $P$ then $P^{\on{op}}$ is a poset with the order inverted. It is a special feature  of PL category that if a poset P is an abstract  spherical  cell complex then $P^{\on{op}}$ is also  an abstract  spherical   cell complex which we call \textit{dual}  to P. This follows from ``PL invariance of a star" theorem. For simplicial triangulations of manifolds it is often called \textit{ Poincaré dual complex}, which shows up in combinatorial proofs of Poincaré duality.
\p{Geometric and abstract subdivisions, cellular local systems.} \label{gas} We follow \cite{Mnev:2007}.
A geometric spherical  cell complex $B_0$ is a\textit{ subdivision} of $B_1$ (or $B_1$ is aggregation of $B_0$ ) if the relative interior of any ball from $B_0$ is contained in the relative interior of a ball from $B_1$. We denote it $B_0 \unlhd B_1$. A geometric aggregation creates poset map $P(B_0) \xar{} P(B_1)$.  A poset map of  abstract spherical  cell complexes is called an \textit{abstract aggregation} if up to PL homeomorphism it can be represented by a geometric aggregation. With the direction of arrows reversed we call such a morphism an \textit{abstract subdivision}.  Thus we get a category $\pmb S^n$ of abstract regular $n$-spherical PL cell complexes and abstract aggregations. We suppose that abstract spherical cell complexes are \textit{oriented}, i.e. for any complex $S$ the fundamental class $[S]$ is chosen and an aggregation map aggregates fundamental the class to the fundamental class.  

An \textit{ abstract spherical cellular local system} $\mathcal S$ on $\pmb B$ is a local system with values in $\pmb S^n$. With $\mathcal S$ we can associate a PL cellular  spherical fiber bundle $\on {{\mathscr T}ot} \mathcal S \xar{} \pmb B$ using iterated cellular cylinders of subdivision maps (this cellular bundle is called ``prismatic" in \cite{Mnev:2007}). This goes as follows. We can realize any chain of abstract aggregations over a simplex  as a chain of geometric aggregations of the geometric realization of the first complex in the chain. Then we can construct  the geometric cellular cylinders of corresponding subdivisions over the simplex (see Fig. \ref{aggr}).  These geometric prismatic trivial bundles constructed separately over each base simplex  can be glued together fiberwise using for PL transition maps the parametric Alexander trick. The result is the PL spherical cellular fiber bundle $\on{\ms T ot} \mc S$.
\begin{figure}[h!]
	\begin{center}
\label{key}
\includegraphics[width=5.0in]{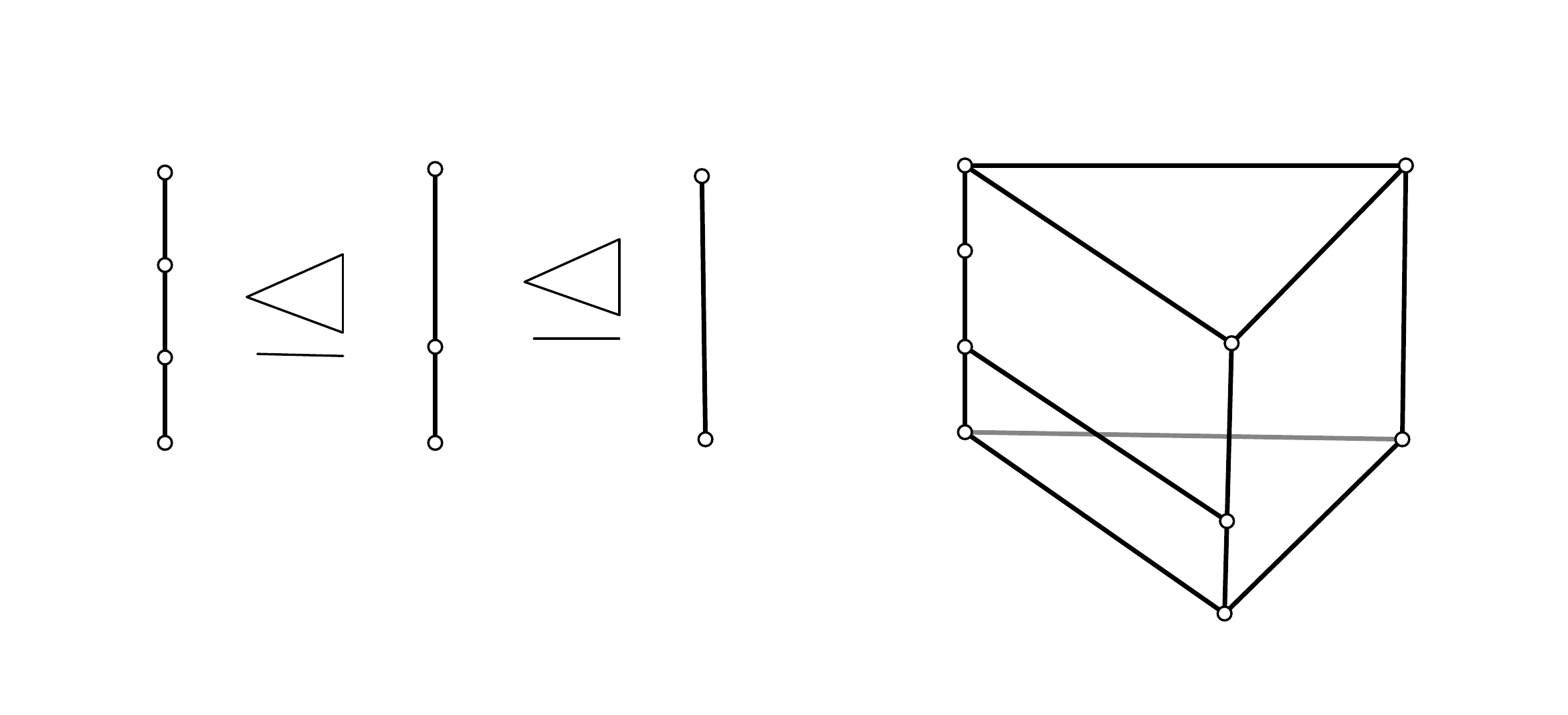}
\caption{\label{aggr}}
\end{center}
\end{figure}

\p{Bundle triangulations versus cellular local systems.}
PL category is defined using triangulations. A spherical PL fiber bundle $S^n\xar{}E \xar{p} B$  has a triangulation by a map of simplicial comlexes $\pmb E \xar{\pmb p} \pmb B$. We call the stalk of triangulation over a base simplex $\sigma^k \in \pmb B(k)$  an elementary simplicial spherical bundle. It triangulates the trivial bundle $S^n \times \Delta^k \xar{\pi} \Delta^k $ in such a way that any simplex in total space is mapped onto a face of the base simplex.  A triangulation  $\pmb E \xar{\pmb p} \pmb B$  is assembled from elementary simplicial spherical bundles using boundary maps - combinatorial  automorphisms of the elementary bundles.
Let us fix the following statement:
 \begin{prop} \label{tri-aggr}
 	A triangulation of a spherical PL fiber bundle $\pmb E\xar{\pmb p} \pmb B$ has a canonically associated spherical cellular local system $\mc S(\pmb p)$ on the first derived subdivision $\on{Sd} \pmb B$ of $\pmb B$  in such a way that the bundle $\on{\mathscr Tot} \mc S(\pmb p)$  is PL isomorphic to $\pmb p$.
 \end{prop}
\begin{proof}
This specifically PL topology statement is based on M. Cohen's theory of transverse cellular maps and corresponding cylinders \cite{Cohen1967}.
 Consider an elementary triangulated $S^n$ bundle over a simplex $\pmb  R\xar{\pmb  q} \Delta^k$. The simplex  $\Delta^k$ has ordered vertices $v_0,....v_k$,  Take an interior point  $x \in \mathop{int} \Delta^k$. Take the simplicial $0$-face $\Delta^{k-1} \xar{\delta_0} \Delta^k$ and a point in the $0$-face $x_0 \in \mathop{int} \Delta^{k-1} $.  The fiber $q^{-1} (x)$ has the structure of abstract regular spherical PL cell complex  $P(x)$ induced from  the triangulation $\pmb  R$. It is a ``multi-simplicial complex''. Its balls are  simplicial prisms (see \cite{Dupont2011}). The prisms are products of simplices. It comes from the fact that the general fiber of a simplicial projection of a simplex onto a simplex is a  product of simplices which are the fibres of the projection over the  vertices in the base. When we move the point  $x$ in the base to the point $x_0$ in the $0$-face all the factors of the prisms in fiber coming from $p^{-1}(v_0)$  shrink to points.
\begin{figure}[h!]
	\begin{center}
		\includegraphics[width=3.0in]{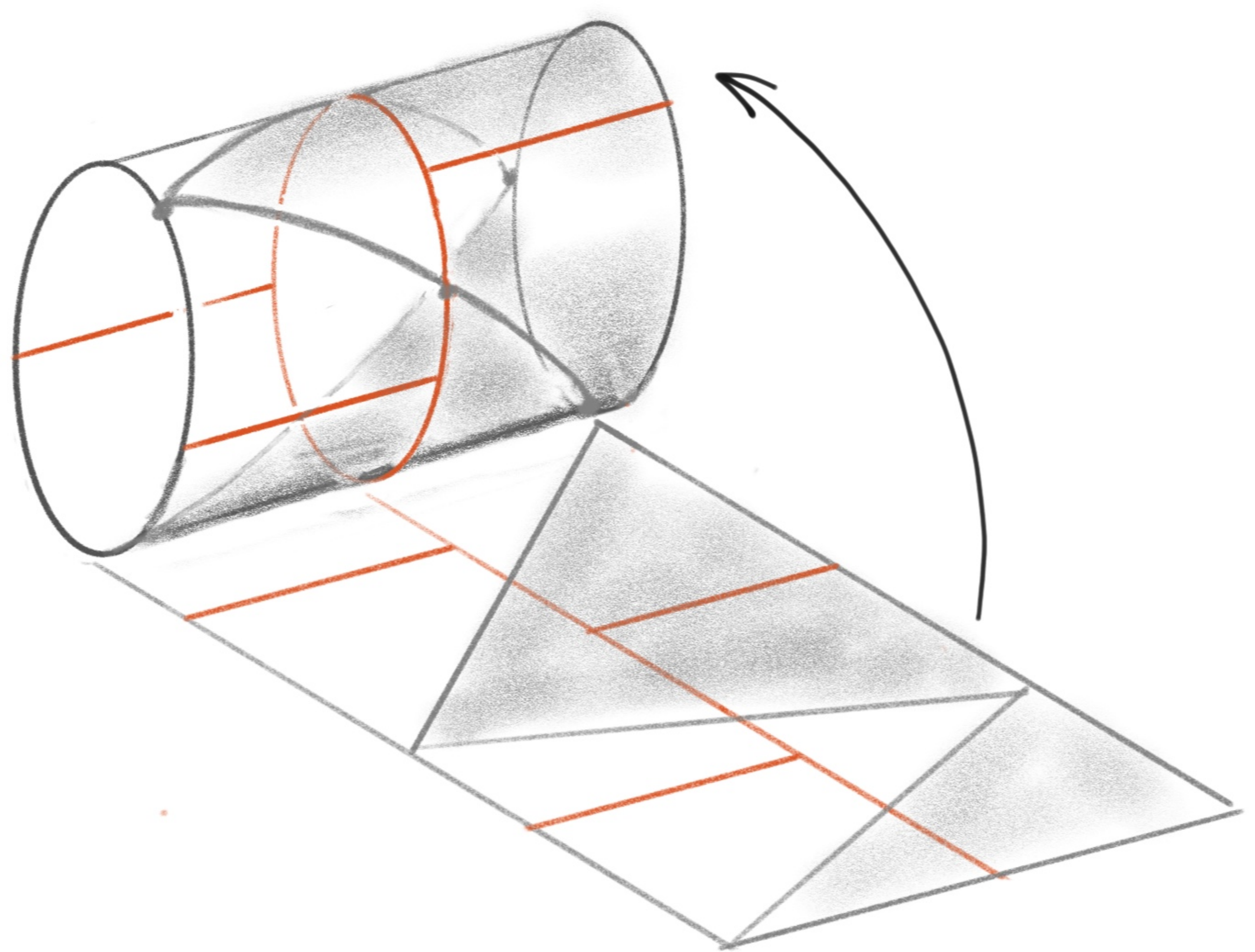} 
		\caption{ Elementary triangulated circle bundle over the interval and the dual pattern of circle subdivisions \label{dt}}
	\end{center}
\end{figure}
This creates multi-simplicial boundary degeneration maps which we consider as poset maps $P(x) \xar{\delta^*_0} P(x_0)$. So  we see over the 1-st derived subdivision of $\Delta^k$ iterated cylinders of those maps.  The key fact is that these boundary degeneration poset maps  are exactly Cohen's \textit{transverse cellular maps} (\cite[Theorem 8.1]{Cohen1967}) -- the poset maps of abstract PL sperical cell complexes which are \textit{dual to aggregation maps}. Thus $P^{\on{op}}(x) \xar{(\delta^*_0)^{\on{op}}} P^{\on{op}}(x_0)$ is an aggregation morphism.
Therefore, over the 1-st derived subdivision $\on{Sd} \Delta^k$ of the base simplex $\Delta^k$ we canonically obtain a diagram $\mc S(\pmb q) $ of aggregations of abstract PL spherical cell complexes (see Fig.  \ref{dt}). Now we may mention that applying Kan derived subdivision functor $\on{Sd} \pmb R \xar{\on{Sd} \pmb q} \on{Sd} \Delta^k$ we obtain a simplicial spherical bundle over $\on {Sd} \Delta^k$ triangulating both elementary bundle $\pmb q$ and the cellular bundle  $\on{\mathscr Tot} \mc S(\pmb q)$. The construction commutes with the assembly of $\pmb p$ from elementary bundles.
\end{proof}
\p{Local formula for the Euler class of a spherical bundle represented by an $\pmb S^n$-local system.} \label{aggreu} By Proposition \ref{tri-aggr} we can functorially  replace  a triangulated spherical bundle on $\pmb B$ by an $\pmb S^n$-local system on $\on{Sd} \pmb B$.  In the language of combinatorics of $\pmb S^n$-local system $\mc S$ on the base $\pmb B$, a rational local formula for Euler class is a rational number associated to the combinatorics of the stalk of the local system $\mc S$ over a $(n+1)$-simplex of the base and representing the simplicial Euler cocycle  for the bundle $|\on{\ms Tot} \mc S| \xar{} |\pmb  B|$.   This stalk is just a chain of $n+1$ aggregations of abstract  cellular spheres, which always can be realized as a chain of geometric aggregations (subdivisions).   The local system $\mc S$ is assembled from stalks over simplices using boundary combinatorial automorphisms of stalks. Therefore, a rational local formula  for the Euler  class in this setting is a rational function of chain of aggregations of spherical cell complexes depending only on the combinatorics of the chain and invariant under automorphisms of boundary subchains. This kind of formula we call an \textit{aggergation local combinatorial formula for the Euler class}. From an aggregation formula we can obtain a simplicial local formula \S  (\ref{slf}) by Proposition \ref{tri-aggr}, integrating the aggregation simplicial  cocycle over the derived subdivision of the base of the elementary simplicial bundle.

\section{Local system of abstract spherical cell complexes as a bigraded model of a spherical fiber bundle} \label{chains}

Assume that we have an abstract $\pmb S^n$-spherical cellular local system $\pmb B \xar{\mc S} \ms N \pmb S^n$ on simplicial locally ordered complex  $\pmb B$. Let us associate to $\mc S$ a local system of  $\on {Ch} S^n$-chain  complexes $$\pmb B \xar{\mc R(\mc S)} \ms N (\on{Ch} S^n)^\op $$ Pick freely the orientations of cells of any complex $S_v, v \in \pmb B(0)$ and form cellular chain complexes  $ R_v = C_\bullet ( S_v)$. Now let $ S_0 \xar{ S(0,1)}  S_1$  be  an aggregation morphism. By definition it is representable by orientation preserving  homeomorphism
$|S_0| \xar{f} |S_1|$ For any closed $k$-ball $B \in |S_1|$, $g^{-1} (B)$ is a union
of closed $k$-balls from $|S_0|$.
We associate to $S(0,1)$ the subdivision chain map\begin{equation}\label{subchain}
 	 R( S_1) \xar{ R( S(0,1))}  R( S_0)
 \end{equation}  sending a $k$-cell from $ R( S_1)$ to the sum of $k$-cells which it aggregates with relative orientations.  By acyclic carriers argument these maps  are quasi-isomorphisms  and they obviously commute with compositions.  The maps on zero-chains  commute with augmentations. The fact that we are in the oriented situation means that the fundamental classes $[S] \in Z_n (S)$ are fixed and the  chain maps $ R( S(0,1))$ sends fundamental  class to fundamental class.

The key but trivial statement of this work is:
\begin{prop}
The cellular chain complex $C_\bullet (\on{\ms  Tot} \mc S)$ is \textit{isomorphic} to  $\on{Tot} (\pmb B, \mc R (\mc S) )$.	
\end{prop}
\begin{proof}
This is an immediate corollary of the construction of ``prismatic bundle" $\on{\ms  Tot} \mc S$  (\S (\ref{gas}), \cite{Mnev:2007}). The cellular differential of the prismatic bundle is naturally decomposed into the sum of vertical and horizontal  differentials.   	
\end{proof}	
But the corollary is that we get an explicit ``bigraded model of fibration."
\begin{cor}
The algebraic bicomplex spectral sequence of the local system $\mc R (\mc S)$ on $\pmb B$ is the Leray-Serre spectral sequence of the PL spherical fiber bundle $|\on{\ms Tot} \mc S |\xar{} |\pmb B|$.
\end{cor}
What follows by Proposition \ref{Igu} is:
\begin{cor} \label{bigrad}
If we endow the local  system $\mc R(\mc S)$ on $\pmb B$ by a system $\mc F$  of SDR on $H(S^n)$ then the expression (\ref{formal}) is an expression for the  simplicial  Euler cocycle on the base of $\pmb B$ of the fiber bundle $|\on{\ms Tot} \mc S|\xar{} |\pmb B|$ represented by the local system $\mc S$ of aggregations of abstract spherical cell complexes.   	
\end{cor}

\section{Combinatorial Hodge-theoretic twisting cochain and a local combinatorial formula for the Euler class of a PL spherical fiber bundle}

Now we may choose SDR of $ R ( S_v)$ on homology and see what happens. There is freedom in interesting choices  but the simplest  is combinatorial  Hodge-theoretic (or Moore-Penrose) matrix retractions.
\p{Rational matrix homology splitting of cellular sphere.} \label{split}
Now our coefficients are rational numbers $\Q$.
Let $S$ be the $n$-dimensional PL cellular sphere. Cells
of $S$ have fixed orientations and are linear ordered in every
dimension. Let  $S$ have $w$  ordered vertices and $f$ ordered top cells.

We can form a based integer cellular chain complex
$C_\bullet(S)$ with differentials $\gamma_i$ represented by  matrices with entries $0, 1, -1$.
Set $ R_\bullet = C_\bullet(S)\otimes \Q$. We have special trivial complex 
$$H(S^n;\Q) =(0\xar{}\Q\xar{}0\xar{}...\xar{}0\xar{}\Q \xar{}0)  $$ It has  $\Q$ as $0$ and $n$ terms, zero at all the other places and zero differential. We denote by $\Q_0,\Q_n$ the two nontrivial modules of $H(S^n;\Q)$. We suppose that $H(S^n;\Q)$ has fixed bases as modules over $\Q$ identified as units in $\Q_0,\Q_n$. The complex $R_\bullet$ has a fixed augmentation
\begin{equation}\label{p_0}
	R_0 \xar{p_0} \Q: p_0 (\beta_1,...,\beta_w) = \sum_{i=1}^v \beta_i, p_0 \gamma_1 =0
\end{equation}
 Fix the cellular fundamental class
$	[S]\in R_n $.
Fixing the fundamental class allows us to define
\begin{equation}\label{i_n}
		\Q \xar{i_n} R_n : i_n (\alpha)=\alpha [A], \gamma_n i_n=0
\end{equation} 
Thus we have identified $R_\bullet$ as an object of $\on{Ch}(S^n; \Q)$.

\bigskip
\noindent
Our aim is to find matrix representation for SDR data of $R_\bullet$ onto $H(S^n;\Q)$.

\bigskip
\noindent
Unwinding conditions  (\ref{sdr}),(\ref{retr}),(\ref{nil}) for the special case of a retraction $R_\bullet$ onto its homology $H(S^n;\Q)$
we got a diagram
\begin{equation}\label{key}
	\xymatrix{
		0  \ar@<.5ex>@{..>}[d]  & 0 \ar@<.5ex>[d] \\
		R_n	 \ar@<.5ex>[r]^{p_n}  \ar@<.5ex>[d]^{\gamma_n}  \ar@<.5ex>@{..>}[u]    &  \Q_n \ar@<.5ex>[u] \ar@<.5ex>[l]^{i_n} \\
		R_{n-1}  \ar@<.5ex>[u]^{F_n} \ar@<.5ex>[d]^{\gamma_{n-1}} &  0 \\
		\vdots   \ar@<.5ex>[u]^{F_{n-1}}  \ar@<.5ex>[d]^{\gamma_{2}} & \vdots \\
		R_1  \ar@<.5ex>[u]^{F_{2}} \ar@<.5ex>[d]^{\gamma_1} & 0 \\
		R_0 \ar@<.5ex>[u]^{F_{1}} \ar@<.5ex>[r]^{p_0} \ar@<.5ex>@{..>}[d]   &  \Q_0 \ar@<.5ex>[l]^{i_0} \ar@<.5ex>[d] \\
		0 \ar@<.5ex>@{..>}[u]     & 0 \ar@<.5ex>[u]
	}
\end{equation}
We can translate axioms (\ref{retr}),(\ref{nil}) in a symmetric way: the two  complexes
\begin{equation}
	\begin{array}{c} \label{ttt} 0 \xar{} \Q_n \xar{i_n}R_n \xar{\gamma_n}R_{n-1} \xar{\gamma_{n-1}}...\xar{\gamma_1}R_0 \xar{p_0} \Q_0 \xar{}0 \\
		0 \xleftarrow{} \Q_n \xleftarrow{p_n}R_n \xleftarrow{F_n}R_{n-1} \xleftarrow{F_{n-1}}...\xleftarrow{F_1}R_0 \xleftarrow{i_0} \Q_0 \xleftarrow{}0
	\end{array}
\end{equation}
are acyclic chain and cochain complexes correspondingly   on the same graded free module over $\Q$, the operators of the second are SDR nullhomotopy operators (i.e. SDR onto zero) for the first. This means that we have identities:
\begin{equation}\label{sdrs}
	\begin{array}{rcll}
		\gamma_1 F_1 + i_0 p_0 & = & Id & \\
		i_np_n  +  F_n\gamma_n & = &Id & \\
		\gamma_j F_j + F_{j-1} \gamma_{j-1}& = &Id & \text{ for } j\neq 1,n\\
	\end{array}
\end{equation}
$$\gamma_{j}\gamma_{j-1}=0, \gamma_n i_n =0,  p_0 \gamma_1 =0, p_n F_n=0, F_1 i_0 =0, F_{j-1}F_{j}=0 $$
For the matrix SDR of the chain sphere $R_\bullet$ on homology we have fixed all the data of the first raw in  (\ref{ttt}).   To get a retraction of $R_\bullet$ onto $H(S^n;\Q)$ we need to find the data in the second row of (\ref{ttt}) satisfying  conditions (\ref{sdrs}) together with the data in the first raw.

\bigskip
\noindent
For a matrix $M$ over $\Q$ denote by $M^\dagger$ its Moore-Penrose inverse matrix.
\begin{lemma}
A 	matrix homology splitting of $R_\bullet$ is provided by data $\la F ,i,p \ra$ where $F_i = \gamma_i^\dagger$, $i_0 = p_0^\dagger$, $p_n = i_n^\dagger$
\end{lemma}	\begin{proof} The Lemma follows from  combinatorial Hodge theory for free based rational chain complex with fixed bases, which provides  a strong deformational retraction $$\la \gamma^\dagger ,i,p \ra $$ of $R_\bullet$ onto homology $H(S^n;\Q)$ (see for example \cite[\S A.1]{Nicolaescu2003}). Here the matix $\gamma^\dag_i$ is the Moore-Penrose inverse of matrix $\gamma_i$, $i_0 = p_0^\dagger$, $p_n = i_n^\dagger$. \end{proof}

We will present matrix  formulas. Our $R_\bullet=C_\bullet (S;\Q)$ is based. We  have canonical scalar products making the bases in $R_\bullet$ orthonormal and therefore we have combinatorial Hodge theory.
Let $$R_{i-1}\xar{\gamma_i^\intercal} R_i $$ be the metric adjoint differential for $\gamma_i$, which is represented  simply by transposed matrix.
Let
\begin{equation}\label{lap}
	\Delta_i = \gamma^\intercal_i \gamma_i + \gamma_{i+1} \gamma_{i+1}^\intercal
\end{equation}
be the combinatorial  Laplace matrix operator.
Our  cellular sphere has $w$ vertices and $f$ top cells.
For the top cell $j$ we denote by $o(j,[A])$ its orientation relatively to the fundamental class. Then the matrix formulas for combinatorial Hodge theory (or Moore-Penrose) homology splitting of $R_\bullet$ are the following. 
Define  matrices $$i_0,i_n, p_0, p_n $$
by the  formulas:
\begin{equation}\label{ip}
	\begin{array}{lrll}
		\Q_0 \xar{i_0} R_0: & i_0(\beta)  &=& \frac{1}{w}\underbrace{(\beta,...,\beta)}_w \\
		\Q_n \xar{i_n} R_n: & i_n (\alpha) &=& \alpha [A] \\
		R_0 \xar{p_0} \Q_0: & p_0(\beta_1,...,\beta_w) &=& \sum_{j=1}^w \beta_j \\
		R_n \xar{p_n} \Q_n: & p_n(\alpha_1,...,\alpha_f) &=& \frac{1}{f}\sum_{j=1}^n o(j,[A])\alpha_j
	\end{array}
\end{equation}
Set
\begin{equation}\label{green}
	G_j = \begin{cases} (\Delta_j + i_j p_j)^{-1} & j=0,1 \\
		\Delta_j^{-1}  & j \neq 0,n \end{cases}
\end{equation}
and set
\begin{equation}\label{F}
	\gamma^\dagger_j=\gamma_j^\intercal G_{j-1} = G_j\gamma_{j-1}^\intercal
\end{equation}

\p{Hodge-theoretic local combinatorial formula for the Euler class of a PL spherical fiber bundle.}

Now we can insert Hodge-theoretic matrix homology splittings of cellular spheres into formula (\ref{formal}) for the formal Euler cocycle and get a theorem.

Suppose we have a chain of spherical aggregations in $\pmb S^n$: $$\mc S =(S_0 \xar{} S_1 \xar{}... \xar{} S_{n+1})$$
Suppose we have a corresponding chain of based chain complexes in $ \on{Ch} (S^n, \Q)$ (see Sec. \ref{chains}):
$$ \mc R (\mc S) = (R_\bullet^0 \xlar{R(0,1)}R_\bullet^1 \xlar{R(1,2)})... \xlar{R(n,n+1)}R_\bullet^{n+1}) $$
where $R(i,i+1)$ are  subdivision chain quasi-isomorphisms (\ref{subchain}).  The differential in $R^i_\bullet$ is denoted by $\gamma$.  For every $R^i_\bullet$  we have Hodge theoretic   
strong deformational retraction onto $H(S^n; \Q)$:
\begin{equation}\label{}
	\begin{tikzcd}
		R^i_\bullet \arrow[r, "p", shift left] \arrow["\gamma^\dagger"', loop, distance=2em, in=215, out=145] & H(S^n;\Q)  \arrow[l, "i", shift left]
	\end{tikzcd}
\end{equation} 
Matrices of all the operators are defined by formulas (\ref{ip}), (\ref{green}), (\ref{F}).   
\begin{theorem} \label{finformula}
For a chain of spherical aggregations $\mc S$,
the rational number obtained as the matrix product 
\begin{equation}\label{last}
	e_{\it CH}(S)  = p_n R(0,1)\gamma^\dagger R(1,2)....\gamma^\dagger R(n,n+1)i_{0}  
\end{equation}
is an aggregation  local combinatorial formula  for the Euler class of PL $S^n$-fiber bundles (as in \S (\ref{aggreu})). 
\begin{proof}
This is the matrix formula from Proposition \ref{Igu} for the twisting cochain in the bigraded model of PL spherical fiber bundle defined by a local system of aggregations (Corollary \ref{bigrad}). It is invariant under  all choices involved, invariant under automorphisms of $R^i_\bullet$, because all the involved Laplace and Green  operators are. It depends up to sign only on the bundle orientation.
\end{proof}

\end{theorem} 

\section{Notes}
We don't know much about the behaviour of the local formula (\ref{last}). To our deep shame we don't know now why the cocycle is a coboundary if $n$ is even. It should be, since the Euler class of even dimensional spherical bundle iz zero. We suspect that the module of number (\ref{last}) has small absolute upper bound, but we don't know it. Meanwhile, we can compare in an interesting way the formula with the formulas for circle bundles 
\cite{Igusa2004, MS}. We postpone it for a later writing.   
Formula (\ref{last}) should have interesting interpretation in terms of cellular combinatorial physics and statistics. It is composed from Moore-Penrose inverses of differentials which have very interesting description \cite[Theorem 5.3]{CCK2015} based on Lothar Berg's theorem \cite{Berg1986} which have no analog yet in differential Hodge theory.

\bibliographystyle{alpha}
\bibliography{euler}
\end{document}